\documentclass[11pt,reqno]{amsart}
\usepackage{amssymb,amsmath,amsthm,enumerate,verbatim,bbm}
\usepackage[a4paper]{geometry}
\usepackage[utf8]{inputenc}
\usepackage[english]{babel}
\usepackage{amsfonts}
\usepackage[pdftex]{graphicx}
\usepackage{mathrsfs}
\usepackage{microtype}
\usepackage{amssymb}
\usepackage{amsthm}
\usepackage{amsmath}
\usepackage{a4wide}
\usepackage{lipsum}
\usepackage{latexsym}
\usepackage{mathrsfs}
\usepackage[v2,tips]{xy}
\usepackage{float}
\usepackage[T1]{fontenc}
\usepackage{lmodern}

\title[When $\mathcal{B}(X)$ is DF but it does not have stable rank one]{Banach spaces whose algebras of operators are Dedekind-finite but they do not have stable rank one}

\author{Bence Horváth}

\newtheorem{Thm}{Theorem}[section]
\newtheorem{Lem}[Thm]{Lemma}
\newtheorem{Prop}[Thm]{Proposition}

\newtheorem{Cor}[Thm]{Corollary}
\newtheorem{Que}[Thm]{Question}

\theoremstyle{definition}
\newtheorem{Def}[Thm]{Definition}
\newtheorem{Ex}[Thm]{Example}
\newtheorem{Rem}[Thm]{Remark}

\newtheorem*{Ack}{Acknowledgement}

\numberwithin{equation}{section}

\DeclareMathOperator*{\Ran}{\mathrm{Ran}}
\DeclareMathOperator*{\Ker}{\mathrm{Ker}}

\DeclareMathOperator*{\inv}{\mathrm{inv}}

\DeclareMathOperator*{\interior}{\mathrm{int}}
\DeclareMathOperator*{\codim}{\mathrm{codim}}

\newcommand{\Addresses}{{% additional braces for segregating \footnotesize
  \bigskip
  \footnotesize

  \textsc{Department of Mathematics and Statistics, Fylde College, Lancaster University, Lancaster, LA1 4YF, United Kingdom}\par\nopagebreak
  \textit{E-mail address}: \texttt{b.horvath@lancaster.ac.uk, horvath@math.cas.cz}
}}

\begin{document}

\maketitle
	
\begin{abstract}
In this note we examine the connection between the stable rank one and Dedekind-finite property of the algebra of operators on a Banach space $X$. We show that for the complex indecomposable but not hereditarily indecomposable Banach space $X_{\infty}$ constructed by Tarbard (\textit{Ph.D. Thesis, University of Oxford,} 2013), the algebra of operators $\mathcal{B}(X_{\infty})$ is Dedekind-finite but does not have stable rank one. While this sheds some light on the Banach space structure of $X_{\infty}$ itself, we observe that the indecomposable but not hereditarily indecomposable Banach space constructed by Gowers and Maurey (\textit{Math. Ann.}, 1997) does not possess this property. We also show that if $K$ is the connected ``Koszmider'' space constructed by Plebanek in ZFC (\textit{Topology and its Applications}, 2004), then $\mathcal{B}(C(K, \mathbb{R}))$ is Dedekind-finite but does not have stable rank one.
\end{abstract}	
	
\section{Introduction and basic terminology}

Let $A$ be a ring. We say that an element $p \in A$ is \textit{idempotent} if $p^2 = p$ holds. Let $p,q \in A$ be idempotents, we say that \textit{$p$ and $q$ are equivalent}, if there exist $a,b \in A$ such that $ab=p$ and $ba=q$. If $p,q \in A$ are equivalent idempotents, we denote this by \textit{$p \sim q$}. It is easy to see that $\sim$ is an equivalence relation on the set of idempotents in $A$. Two idempotents $p,q \in A$ are \textit{orthogonal} if $pq=0=qp$.
\begin{Def}
Let $A$ be a unital ring with identity $1_A$. Then $A$ is called
\begin{enumerate}
\item \textit{Dedekind-finite} or \textit{directly finite} or \textit{DF} for short, if the only idempotent $p \in A$ with $p \sim 1_A$ is the identity $1_A$,
\item \textit{Dedekind-infinite} if it is not Dedefind-finite,
\item \textit{properly infinite} if there exist orthogonal idempotents $p,q \in A$ such that $p,q \sim 1_A$.
\end{enumerate}
\end{Def}
It is easy to see that a properly infinite ring is Dedekind-infinite. Clearly every commutative, unital ring is Dedekind-finite. Another easy example is the matrix ring $M_n(\mathbb{C}), \, (n \geq 1)$ since an $(n \times n)$ complex matrix is left-invertible if and only if it is right-invertible. \\
Therefore it is natural to examine the unital Banach algebra $\mathcal{B}(X)$ from this perspective, where $\mathcal{B}(X)$ denotes the \textit{bounded linear operators} on an infinite-dimensional Banach space $X$. In this note all Banach spaces are assumed to be complex, unless explicitly stated otherwise. The systematic study of Dedekind-(in)finiteness of $\mathcal{B}(X)$ was laid out by Laustsen in \cite{ringfinofopalgs}, where the author characterises Dedekind-finiteness and properly infiniteness of $\mathcal{B}(X)$ in terms of the complemented subspaces of $X$. For our purposes the former is of greater importance, therefore we recall this result here:
\begin{Lem}\label{dfequiv}(\cite[Corollary~1.5]{ringfinofopalgs})
Let $X$ be a Banach space. Then $\mathcal{B}(X)$ is Dedekind-finite if and only if no proper, complemented subspace of $X$ is isomorphic to $X$ as a Banach space.
\end{Lem}
Let us recall that an infinite-dimensional Banach space $X$ is \textit{indecomposable}, if there are no closed, infinite-dimensional subspaces $Y,Z$ of $X$ such that $X$ can be written as the topological direct sum $X = Y \oplus Z$, this latter being a short-hand notation for $X= Y + Z$ and $Y \cap Z = \lbrace 0 \rbrace$. A Banach space $X$ is \textit{hereditarily indecomposable} (or \textit{HI} for short) if every closed, infinite-dimensional subspace of $X$ is indecomposable.
As it is observed in \cite[Corollary~1.7]{ringfinofopalgs}, every hereditarily indecomposable Banach space $X$ satisfies the conditions of Lemma \ref{dfequiv}. However, as we shall demonstrate in Corollary \ref{hiissr1}, if $X$ is an HI space, then $\mathcal{B}(X)$ in fact possesses the stronger property of having stable rank one. This definition was introduced by Rieffel in \cite{rieffel}:
\begin{Def}
A unital Banach algebra $A$ has \textit{stable rank one} if the group of invertible elements $\inv(A)$ is dense in $A$ with respect to the norm topology.
\end{Def}	%%%%%%%%%%%%%%%%%%%%%%%%%%%%%%%%%%%%%%%%%%%%%%%%%%%%%%%%%%%%%%%%%%%%%%%%%%%%%%%%%%%%%%%%%%%%%%%%%%%%%%%%%%%%%%%%%%%%%%%%%%%%%%%%%%	%%%%%%%%%%%%%%%%%%%%%%%%%%%%%%%%%%%%%%%%%%%%%%%%%%%%%%%%%%%%%%%%%%%%%%%%%%%%%%%%%%%%%%%%%%%%%%%%%%%%%%%%%%%%%%%%%%%%%%%%%%%%%%%%%%%	%%%%%%%%%%%%%%%%%%%%%%%%%%%%%%%%%%%%%%%%%%%%%%%%%%%%%%%%%%%%%%%%%%%%%%%%%%%%%%%%%%%%%%%%%%%%%%%%%%%%%%%%%%%%%%%%%%%%%%%%%%%%%%%%%%%

\section{Algebras of operators with stable rank one and their connection to Dedekind-finiteness}

The following observation is an immediate corollary of \cite[Proposition~3.1]{rieffel}, we include a short proof for the reader's convenience.	

\begin{Lem}\label{dfisstr1}
A unital Banach algebra with stable rank one is Dedekind-finite.
\end{Lem}

\begin{proof}
Let $A$ be a Banach algebra with stable rank one. Assume $p \in A$ is idempotent such that $p \sim 1_A$, then there exist $a,b \in A$ such that $p=ab$ and $1_A = ba$. Let $u \in \inv(A)$ be such that $\Vert a-u \Vert < \Vert b \Vert^{-1}$, then $\Vert 1_A -bu \Vert = \Vert ba - bu \Vert \leq \Vert b \Vert \Vert a-u \Vert <1$. So in particular $bu \in \inv(A)$ holds, and consequently $b=buu^{-1} \in \inv(A)$. From this and $1_A = ba$ we get $a=b^{-1}$, consequently $p=ab=1_A$. Thus $A$ is Dedekind-finite.
\end{proof}

Note however that the converse of the previous lemma is clearly false. We demonstrate this with an example which will be essential in the proof of our main result, Theorem \ref{main}.

Let us recall that in a Banach algebra $A$ an element $a \in A$ is a \textit{topological zero divisor} if $\inf \lbrace \Vert xa \Vert + \Vert ax \Vert : \, x \in A, \Vert x \Vert = 1 \rbrace = 0$. It is a standard result from the fundamental theory of Banach Algebras, (see for example \cite[Chapter~2,~Theorem~14]{Bonsall}) that for a unital Banach algebra $A$ the \textit{topological boundary of $\inv(A)$}, that is $\partial(\inv(A)):= \overline{\inv(A)} \backslash \inv(A)$, is contained in the set of topological zero divisors of $A$. 
\\
In what follows $\mathbb{N}$ denotes the natural numbers excluding zero, and $\mathbb{N}_0 := \mathbb{N} \cup \lbrace 0 \rbrace$.

\begin{Ex}\label{exampleell1}
The unital Banach algebra $\ell_1(\mathbb{N}_0)$ (endowed with the convolution product $(a_n)_{n=0}^{\infty} * (b_n)_{n=0}^{\infty} := \sum\limits_{m=0}^{\infty} \sum\limits_{k=0}^m a_k b_{m-k}$) is Dedekind-finite but does not have stable rank one. The former is trivial since $\ell_1(\mathbb{N}_0)$ is commutative. Now let us show that it does not have stable rank one. This in fact is contained in the proof of \cite[Proposition~4.7]{dragakania}, we include the argument here for the sake of completeness. Let $(\delta_n)_{n \in \mathbb{N}_0}$ stand for the canonical basis of $\ell_1(\mathbb{N}_0)$, clearly $\delta_0$ is the identity in $\ell_1(\mathbb{N}_0)$. Observe that $\delta_1$ is a non-invertible element in $\ell_1(\mathbb{N}_0)$. We now show that $\delta_1$ is not a topological zero divisor. To see this, let $x= (x_n)_{n=0}^{\infty} \in \ell_1(\mathbb{N}_0)$ be arbitrary. Then
\begin{align}
\Vert x \ast \delta_1 \Vert = \left\Vert \sum\limits_{n \in \mathbb{N}} x_{n-1} \delta_n \right\Vert = \sum\limits_{n \in \mathbb{N}_0} \vert x_n \vert = \Vert x \Vert.
\end{align}
Thus by the discussion preceeding the example we see that $\delta_1 \notin \partial(\inv(\ell_1(\mathbb{N}_0)))$. Hence we conclude that $\delta_1 \notin \overline{\inv(\ell_1(\mathbb{N}_0))}$, therefore $\ell_1(\mathbb{N}_0)$ cannot have stable rank one.
\end{Ex}

As we will see in Corollary \ref{hiissr1}, all the examples given in \cite{ringfinofopalgs} such that $\mathcal{B}(X)$ is Dedekind-finite have stable rank one. Thus the following question naturally arises:
\begin{Que}
Does there exist a Banach space $X$ such that $\mathcal{B}(X)$ is Dedekind-finite but it does not have stable rank one?
\end{Que}
The purpose of the following is to answer this question in the positive.
\\
Recall that if $A$ is a unital algebra over a field $K$ and $C$ is a unital subalgebra then $\inv(C) \subseteq \inv(A) \cap C$ holds but there is not equality in general. In the following, if $J$ is a two-sided ideal of $A$ we introduce the notation $\tilde{J}:= K1_A + J$. 

\begin{Lem}\label{subalginv}
Let $A$ be an algebra over a field $K$ and let $J \trianglelefteq A$ be a proper, two-sided ideal. Then for the unital subalgebra $\tilde{J}$ the equality $\inv(\tilde{J})= \inv(A) \cap \tilde{J}$ holds.
\end{Lem}

\begin{proof}
It is clear that $\tilde{J}$ is a unital subalgebra of $A$. Thus we only need to show the inclusion $\inv(A) \cap \tilde{J} \subseteq \inv(\tilde{J})$. To see this let us pick an arbitrary $\lambda \in K$ and $j \in J$ such that $\lambda 1_A + j \in \inv(A)$. Clearly $\lambda \neq 0$ otherwise $j \in \inv(A)$ which contradicts $J$ being a proper subset of $A$. Now it is clear that $a:= \lambda^{-1}1_A - \lambda^{-1}(\lambda 1_A + j)^{-1} j \in K1_A + J$, and a simple calculation shows that $a(\lambda 1_A + j) =1_A = (\lambda 1_A + j)a$ holds, proving the claim.
\end{proof}

\begin{Rem}\label{unitasubalg}
If $A$ is a complex unital Banach algebra and $J \trianglelefteq A$ is a proper, closed, two-sided ideal in $A$ then $\tilde{J}:= \mathbb{C} 1_A + J$ is a closed, unital subalgebra of $A$. (Closedness follows from the fact that $\mathbb{C} 1_A$ and $J$ are respectively finite-dimensional and closed subspaces of the Banach space $A$.) Also, $\tilde{J}$ is equal to the closed unital subalgebra in $A$ generated by the set $\lbrace 1_A \rbrace \cup J$.
\end{Rem}

\begin{Lem}\label{emptyint}
Let $A$ be a complex, unital Banach algebra and let $a \in A$ be such that $0 \in \mathbb{C}$ is not in the interior of the spectrum $\sigma_A(a)$. Then $a \in \overline{\inv(A)}$.
\end{Lem} 

\begin{proof}
By the hypothesis it follows that $0 \notin \interior(\sigma_A(a)) = \mathbb{C} \backslash \left( \overline{\mathbb{C} \backslash \sigma_A(a)} \right)$. Thus there exists a sequence $(\lambda_n)_{n \in \mathbb{N}}$ in the resolvent set of the element $a$ converging to $0 \in \mathbb{C}$. Therefore $(a-\lambda_n 1_A)_{n \in \mathbb{N}}$ is a sequence of invertible elements in $A$ converging to $a$.
\end{proof}

An operator $T \in \mathcal{B}(X)$ is called \textit{inessential} if for any $S \in \mathcal{B}(X)$ it follows that $\dim(\Ker(I_X - ST)) < \infty$ and $\codim_X(\Ran(I_X -ST)) < \infty$. The set $\mathcal{E}(X)$ of inessential operators forms a proper, closed, two-sided ideal of $\mathcal{B}(X)$, see \cite[Remark~4.3.5]{pietsch}.

\begin{Prop}\label{subalgsr1}
Let $X$ be a Banach space, and let $J \trianglelefteq \mathcal{B}(X)$ be a closed, two-sided ideal with $J \subseteq \mathcal{E}(X)$. Then for any $\alpha \in \mathbb{C}$ and $T \in J$, $\alpha I_X + T \in \overline{\inv(\tilde{J})}$ holds, and therefore $\tilde{J}$ has stable rank one.
\end{Prop}

\begin{proof}
Let us pick $\alpha \in \mathbb{C}$ and $T \in J$ arbitrary. It is an immediate corollary of Lemma \ref{subalginv} that $\sigma_{\tilde{J}}(T) = \sigma_{\mathcal{B}(X)}(T)$. Now by the Spectral Mapping Theorem $\sigma_{\tilde{J}}(\alpha I_X + T) = \alpha + \sigma_{\tilde{J}}(T)$,
putting this together with the previous we conclude that
\begin{align}
\sigma_{\tilde{J}}(\alpha I_X + T) = \alpha + \sigma_{\mathcal{B}(X)}(T).
\end{align}
Since $T \in J \subseteq \mathcal{E}(X)$, it follows from \cite[Lemma~5.6.1]{caradus} that $T$ is a \textit{Riesz operator} (see \cite[Definition~3.1.1]{caradus}), thus $\sigma_{\mathcal{B}(X)}(T) \backslash \lbrace 0 \rbrace$ has no accumulation point, thus $\sigma_{\mathcal{B}(X)}(T)$ must be countable. Consequently $\sigma_{\tilde{J}}(\alpha I_X + T)$ must be countable, thus it has empty interior, so in particular Lemma \ref{emptyint} yields $\alpha I_X + T \in \overline{\inv(\tilde{J})}$.
\end{proof}

\begin{Rem}
Let us note that in the previous proposition the assumption that the ideal is contained in the inessential operators cannot be dropped in general. To see this we consider the $p^{\text{th}}$ quasi-reflexive James space $\mathcal{J}_p$, where $1<p< \infty$. Since the closed, two-sided ideal $\mathcal{W}(\mathcal{J}_p)$ of weakly compact operators is one-codimensional in $\mathcal{B}(\mathcal{J}_p)$, it is in particular a complemented subspace of $\mathcal{B}(\mathcal{J}_p)$ and therefore $\mathcal{B}(\mathcal{J}_p)= \mathbb{C}I_{\mathcal{J}_p} + \mathcal{W}(\mathcal{J}_p)$ holds. On the other hand, as observed in \cite[Propostition~1.13]{ringfinofopalgs}, the Banach algebra $\mathcal{B}(\mathcal{J}_p)$ is Dedekind-infinite so by Lemma \ref{dfisstr1} it cannot have stable rank one.
\end{Rem}

On a Banach space $X$ an operator $T \in \mathcal{B}(X)$ is called \textit{strictly singular} if there is no infinite-dimensional subspace $Y$ of $X$ such that $T \vert_Y$ is an isomorphism onto its range. The set of strictly singular operators on $X$ is denoted by $\mathcal{S}(X)$ and it is a closed, two-sided ideal in $\mathcal{B}(X)$. By \cite[Theorem~5.6.2]{caradus} the containment $\mathcal{S}(X) \subseteq \mathcal{E}(X)$ also holds.
\begin{Cor}\label{hiissr1}
For a complex hereditarily indecomposable Banach space $X$ the Banach algebra $\mathcal{B}(X)$ has stable rank one.
\end{Cor}

\begin{proof}
As it was proven by Gowers and Maurey in \cite[Theorem~18]{GM0}, for any complex HI space $\mathcal{B}(X) = \mathbb{C} I_X + \mathcal{S}(X)$ holds. Together with Proposition \ref{subalgsr1} the result immediately follows.
\end{proof}
The result above is known, see for example in \cite{dragakania} and the text preceeding Theorem 4.16, although deduced in a slightly different way to ours.

The following simple algebraic lemma is the key step in the proof our main result.
\begin{Lem}\label{hereditarydf}
Let $A$ be a unital algebra over a field $K$ and let $J \trianglelefteq A$ be a two-sided ideal such that both $\tilde{J}$ and $A/J$ are Dedekind-finite. Let $\pi: \, A \rightarrow A/J$ denote the quotient map. If $\pi \left[ \inv(A) \right] = \inv \left( A/J \right)$ holds then $A$ is Dedekind-finite.
\end{Lem}

\begin{proof}
Let $p \in A$ be an idempotent such that $p \sim 1_A$. Then there exist $a,b \in A$ such that $ab=1_A$ and $ba=p$. The identities $\pi(a) \pi(b) = \pi(1_A)$ and $\pi(b) \pi(a)= \pi(p)$ show that $\pi(p)$ is an idempotent in $A/J$ such that $\pi(p) \sim \pi(1_A)$. Since $A/J$ is DF it follows that $\pi(p)=\pi(1_A)$, equivalently $\pi(b) \pi(a) = \pi(1_A)$ and consequently $\pi(a) \in \inv \left( A/J \right)$. By the assumption $\pi$ restricted to the set of invertible elements of $A$ surjects onto the set of invertible elements of $A/J$ therefore there exists $c \in \inv(A)$ such that $\pi(a) = \pi(c)$, equivalently $a-c \in J$. Thus $c^{-1} -b = c^{-1}ab- c^{-1}cb = c^{-1}(a-c)b \in J$. Let us define $a':=(a-c)c^{-1}$ and $b':=c(b-c^{-1})$, it is clear from the previous that $a',b' \in J$. Now we show that the following identities hold:
\begin{itemize}
\item $(1_A+a')(1_A+b')=1_A$ or equivalently $a'+b'+a'b'=0$,
\item $(1_A +b')(1_A +a')= cpc^{-1}$ or equivalently $a'+b'+b'a'=cpc^{-1}-1_A$.
\end{itemize}
To see these, we observe that from the definitions of $a'$ and $b'$ we obtain
\begin{align}
a'+b'=(a-c)c^{-1} + c(b-c^{-1})= ac^{-1} + cb -2 \cdot 1_A,
\end{align}
\begin{align}
b'a' = c(b-c^{-1})(a-c)c^{-1} = c(ba-bc-c^{-1}a +1_A)c^{-1} = cpc^{-1} - cb -ac^{-1} +1_A,
\end{align}
\begin{align}
a'b' = (a-c)c^{-1}c(b-c^{-1}) = ab -ac^{-1} -cb +1_A = 2 \cdot 1_A -ac^{-1} -cb.
\end{align}
The above immediately yield the required identities. Thus we obtained that $cpc^{-1}$ is an idempotent in $\tilde{J}$ equivalent to $1_A$. Since $\tilde{J}$ is DF it follows that $cpc^{-1}= 1_A$. This is, $p=1_A$ which concludes the proof.
\end{proof}
In what follows, if $K$ is a compact Hausdorff space then $C(K)$ denotes the complex valued continuous functions on $K$. If $X$ is a Banach space then $\mathcal{K}(X)$ denotes the closed, two-sided ideal of compact operators. By \cite[Theorems~4.4.4 and 5.6.2]{caradus} the containment $\mathcal{K}(X) \subseteq \mathcal{E}(X)$ holds.
\begin{Rem}
Let us note here that in the previous lemma, the condition that the invertible elements in $A$ surject onto the invertible elements in $A/J$ is not superfluous. To see this, we recall some basic properties of the \textit{Toeplitz algebra}, see \cite[Example~9.4.4]{rordam} for full details of the construction. Let $\mathcal{H}$ be a separable Hilbert space and let $S \in \mathcal{B}(\mathcal{H})$ be the right shift operator, let $S^* \in \mathcal{B}(\mathcal{H})$ denote its adjoint. The unital sub-$C^*$-algebra of $\mathcal{B}(\mathcal{H})$ generated by $S$ is called the Toeplitz-algebra $\mathcal{T}$. We recall that $\mathcal{K}(\mathcal{H}) \subseteq \mathcal{T}$ and that $\mathcal{T}/ \mathcal{K}(\mathcal{H})$ is isomorphic to $C(\mathbb{T})$, where $\mathbb{T}$ is the unit circle. Since $C(\mathbb{T})$ is commutative, it is clearly Dedekind-finite. As is well-known, (see \cite[Corollary~5]{corlar} or by Proposition \ref{subalgsr1} above) $\tilde{\mathcal{K}}(\mathcal{H})$ has stable rank one thus by Lemma \ref{dfisstr1} it is also Dedekind-finite. On the other hand, $S^* S = I_{\mathcal{H}}$ and $SS^* \neq I_{\mathcal{H}}$, thus $\mathcal{T}$ is Dedekind-infinite.
\end{Rem}

For a unital Banach algebra $A$ let $\exp(A):= \lbrace \exp(a): \, a \in A \rbrace$. Recall that $\exp(A) \subseteq \inv(A)$ and when $A$ is commutative, $\exp(A)$ is both a subgroup and the connected component of the identity in $\inv(A)$. In other words, $\exp(A)$ is the maximal connected subset of $\inv(A)$ -ordered by inclusion- containing $1_A$. For further details we refer the reader to \cite[Corollary~2.4.27]{Dales}.
\begin{Lem}\label{surjifconn}
Let $A$ be a unital Banach algebra and suppose $J \trianglelefteq A$ is a closed, two-sided ideal in $A$ such that $A/J$ is commutative. Let $\pi: \, A \rightarrow A/J$ denote the quotient map. If $\inv (A/J)$ is connected then $\pi \left[ \exp(A) \right] = \inv(A/J)$ holds. In particular $\pi \left[ \inv(A) \right] = \inv(A/J)$.
\end{Lem}

\begin{proof}
Since $A/J$ is commutative and $\inv(A/J)$ is connected it follows that $\inv(A/J) = \exp(A/J)$. We now observe that $\exp(A/J) = \pi[\exp(A)]$ holds, since for any $a \in A$, the series expansion of $\exp(a)$ converges (absolutely) in $A$ and the quotient map $\pi$ is a continuous surjective algebra homomorphism; thus it readily follows that $\pi(\exp(a)) = \exp(\pi(a))$. The second part of the claim follows from $\pi[\inv(A)] \subseteq \inv(A/J)$.
\end{proof}

\begin{Lem}\label{ellinvconnected}
The group $\inv(\ell_1(\mathbb{N}_0))$ is connected.
\end{Lem}

\begin{proof}
Let $A:= \ell_1(\mathbb{N}_0)$. It is known (see for example \cite[Theorem~4.6.9]{Dales}) that the character space $\Gamma_A$ of $A$ is homeomorphic to the closed unit disc $\overline{\mathbb{D}}$. Thus by the Arens--Royden Theorem (see \cite[3.5.19~Theorem]{Palmer} and the text preceeding it) we obtain the following isomorphism of groups:
\begin{align}
\inv(A) / \exp(A) \simeq \inv(C(\overline{\mathbb{D}})) / \exp(C(\overline{\mathbb{D}})) \simeq \pi^1(\overline{\mathbb{D}}),
\end{align}
where $\pi^1(\overline{\mathbb{D}})$ denotes the first fundamental group of $\overline{\mathbb{D}}$. Since $\overline{\mathbb{D}}$ is simply connected we obtain $\inv(A) = \exp(A)$ proving that $\inv(A)$ is connected as required.
\end{proof}

\begin{Rem}
In the proof of the previous lemma we do not use the surjective part of the Arens--Royden Theorem, only the much weaker statement that $\inv(A) / \exp(A)$ injects into $\inv(C(\Gamma_A)) / \exp(C(\Gamma_A))$.
\end{Rem}
Let us recall the properties of Tarbard's ingenious indecomposable Banach space construction that are relevant to our purposes, we refer the interested reader to \cite[Chapter~4]{Tarbard} to see the following theorem in its full might.
\begin{Thm}(\cite[Theorem~4.1.1]{Tarbard})\label{tarbard}
There exists an indecomposable Banach space $X_{\infty}$ such that the unital Banach algebras $\mathcal{B}(X_{\infty}) / \mathcal{K}(X_{\infty})$ and $\ell_1(\mathbb{N}_0)$ are isometrically isomorphic.
\end{Thm}
We are now ready to state and prove the main result of this note.
\begin{Thm}\label{main}
The Banach algebra $\mathcal{B}(X_{\infty})$ is Dedekind-finite but does not have stable rank one.
\end{Thm}

\begin{proof}
We first show that $B(X_{\infty})$ does not have stable rank one. Assume towards a contradiction that it does. Then it immediately follows that $\mathcal{B}(X_{\infty}) / \mathcal{K}(X_{\infty})$ also has stable rank one, which in view of Theorem \ref{tarbard} is equivalent to $\ell_1(\mathbb{N}_0)$ having stable rank one. This is impossible by Example \ref{exampleell1}. \\
Now we show that $\mathcal{B}(X_{\infty})$ is Dedekind-finite. By Proposition \ref{subalgsr1} we obtain that $\tilde{\mathcal{K}}(X_{\infty})$ has stable rank one so by Lemma \ref{dfisstr1} it is Dedekind-finite. By Example \ref{exampleell1} we have that $\ell_1(\mathbb{N}_0)$ and thus $\mathcal{B}(X_{\infty}) / \mathcal{K}(X_{\infty})$ is also Dedekind-finite. Thus applying Lemmas \ref{ellinvconnected}, \ref{surjifconn} and \ref{hereditarydf} successively, we obtain that $\mathcal{B}(X_{\infty})$ is Dedekind-finite, which completes the proof.
\end{proof}

With the aid of Lemma \ref{dfequiv} we observe the following:
  
\begin{Cor}
No proper, complemented subspace of $X_{\infty}$ is isomorphic to $X_{\infty}$.
\end{Cor}

We do not know if there is an entirely Banach space-theoretic proof of this result. 
\\
However, we would like to draw the reader's attention to the fact that the previous corollary does not hold in general for indecomposable Banach spaces. This follows directly from a deep result of Gowers and Maurey \cite{GM}. 
\\
We recall that an infinite-dimensional Banach space $X$ is \textit{prime} if it is isomorphic to all its infinite-dimensional, complemented subspaces.
\begin{Thm}(\cite[Section~(4.2) and Theorem~13]{GM})\label{GM}
There exists an indecomposable, prime Banach space.
\end{Thm}

In fact, with the help of two easy lemmas we can say a bit more. In order to do this let us recall the following well-known result, see for example the second part of \cite[Corollary~1.5]{ringfinofopalgs}:
\begin{Lem}\label{idempeqbanach}
Let $X$ be a Banach space, let $P,Q \in \mathcal{B}(X)$ be idempotents. Then $P \sim Q$ if and only if $\Ran(P) \simeq \Ran(Q)$.
\end{Lem}

\begin{Lem}\label{indecnotpropinf}
Let $X$ be an indecomposable Banach space. Then $\mathcal{B}(X)$ cannot be properly infinite.
\end{Lem}

\begin{proof}
Assume towards a contradiction that $\mathcal{B}(X)$ is properly infinite. Then there exist $P,Q \in \mathcal{B}(X)$ orthogonal idempotents such that $P,Q \sim I_X$. By Lemma \ref{idempeqbanach} this is equivalent to $\Ran(P) \simeq X \simeq \Ran(Q)$. Clearly $X = \Ran(P) \oplus \Ran(I_X-P)$ and since $\Ran(P)$ is infinite-dimensional and $X$ is indecomposable we obtain that $\Ran(I_X -P)$ must be finite-dimensional. Consequently, the range of $Q = Q(I_X - P)$ is finite-dimensional, contradicting $\Ran(Q) \simeq X$.
\end{proof}
An infinite-dimensional Banach space $X$ is \textit{primary} if for every $P \in \mathcal{B}(X)$ idempotent either $\Ker(P)$ or $\Ran(P)$ is isomorphic to $X$. A prime Banach space is clearly primary.
\begin{Lem}\label{primarynotdf}
Let $X$ be a primary Banach space. Then $\mathcal{B}(X)$ is Dedekind-infinite.
\end{Lem}

\begin{proof}
Let $P \in \mathcal{B}(X)$ be an idempotent with $\dim(\Ker(P))=1$. Since $X$ is primary, $\Ran(P) \simeq X$ holds. By Lemma \ref{idempeqbanach} this is equivalent to $P \sim I_X$. If $\mathcal{B}(X)$ were DF then $P = I_X$ which is impossible.
\end{proof}

Theorem \ref{GM} ensures that the following corollary of Lemmas \ref{indecnotpropinf}, \ref{primarynotdf} is not vacuous:

\begin{Cor}
For an indecomposable, primary Banach space $X$ the algebra of operators $\mathcal{B}(X)$ is Dedekind-infinite but not properly infinite.
\end{Cor}

\begin{Rem}
	Let $K$ be a countable compact metric space. By a deep result \cite[Theorem~B]{MPZ} of Motakis, Puglisi and Zisimopoulou, there exists a Banach space $X_K$ such that $\mathcal{B}(X_K) / \mathcal{K}(X_K) \simeq C(K)$. In \cite{rieffel}, Rieffel introduced the notion of \textit{left stable rank} in general for unital Banach algebras. We observe here that $\mathcal{B}(X_K)$ has left stable rank $1$ or $2$, in the sense of Rieffel. By \cite[Corollary~4.12]{rieffel}, to show this, it is enough to observe that $\tilde{\mathcal{K}}(X_K)$ and $C(K)$ have both stable rank one. The former follows from Proposition \ref{subalgsr1}. For the latter, in view of \cite[Proposition~1.7]{rieffel}, it is enough to see that $K$ is zero-dimensional. This in particular follows if $K$ is totally disconnected, by \cite[Proposition~3.1.7]{Arh}. But this is clearly holds, since $K$ is countable and compact. We do not know however if for certain $K$-s the Banach algebra $\mathcal{B}(X_K)$ is DF but does not have stable rank one; our ``lifting invertible elements'' method is not applicable here, since $K$ is totally disconnected.
\end{Rem}

After reading the first draft of this paper, it was suggested to us by Piotr Koszmider that the $C(K)$-space from \cite{plebanek} is an example of a \textit{real} Banach space such that its algebra of operators is Dedekind-finite but it does not have stable rank one. With his kind permission we include a proof here.

In the following, if $K$ is a compact Hausdorff space, we always consider $C(K)$ as a \textit{real} Banach space. To emphasise this, we write $C(K, \mathbb{R})$ for $C(K)$. Let $K$ be a compact Hausdorff space, let $g \in C(K, \mathbb{R})$ then 
\begin{align}
	M_g : \, C(K, \mathbb{R}) \rightarrow C(K, \mathbb{R}); \quad f \mapsto fg
\end{align}	
is called the \textit{multiplication} operator. An operator $T \in \mathcal{B}(C(K, \mathbb{R}))$ is called a \textit{weak multiplication} if there is a $g \in C(K, \mathbb{R})$ and $S \in \mathcal{S}(C(K, \mathbb{R}))$ such that $T = M_g + S$. \\
We say that an infinite compact Hausdorff space is a \textit{Koszmider space} if every bounded linear operator  on $C(K, \mathbb{R})$ is a weak multiplication. In \cite[Theorem~6.1]{koszmider1} Koszmider showed that assuming the Continuum Hypothesis, both connected and zero-dimensional Koszmider spaces exist. In \cite[Theorem~1.3]{plebanek} Plebanek showed the existence of a connected Koszmider space without any assumptions beyond ZFC. We recall the following two results on Koszmider spaces:

\begin{Thm}(\cite[Theorem~6.5(i)]{dkkkl})\label{koszmiderweakcalkin}
	Let $K$ be a Koszmider space without isolated points. Then $\mathcal{B}(C(K, \mathbb{R})) / \mathcal{S}(C(K, \mathbb{R}))$ and $C(K, \mathbb{R})$ are isomorphic as unital Banach algebras.
\end{Thm}

The following result is an immediate corollary of Theorems 2.3 and 5.2 in Koszmider's paper \cite{koszmider1}.

\begin{Thm}(\cite{koszmider1})\label{opalgofckoszmiderdf}
	Let $K$ be a connected Koszmider space. Then $C(K, \mathbb{R})$ is not isomorphic to any of its proper, closed subspaces.
\end{Thm}

\begin{Lem}\label{realcknotsr1}
	Let $K$ be a compact connected Hausdorff space with at least two points. Then $C(K, \mathbb{R})$ does not have stable rank one.
\end{Lem}

\begin{proof}
	Let $x,y \in K$ be distinct. Since $K$ is Hausdorff, we can take $U,V$ disjoint open subsets of $K$ such that $x \in U$ and $y \in V$. By Urysohn's Lemma there exist $f,g \in C(K, \mathbb{R})$ supported on $U$ and $V$, respectively, such that $f(x)=1$ and $g(y)=1$. Let $h:= f - g$, clearly $h \in C(K,\mathbb{R})$ is such that $h(x)=1$ and $h(y)=-1$. Suppose $k \in C(K, \mathbb{R})$ is such that $\Vert k - h \Vert < 1/2$, thus $\vert k(x) -1 \vert < 1/2$ and $\vert k(y) +1 \vert < 1/2$. In particular, $1/2 < k(x)$ and $-1/2 > k(y)$, thus by continuity of $k$ and connectedness of $K$ we obtain that there is $z \in K$ such that $k(z) =0$. Thus $k \in C(K, \mathbb{R})$ cannot be invertible. This shows that $\inv(C(K, \mathbb{R}))$ is not dense in $C(K, \mathbb{R})$, as required.
\end{proof}

\begin{Rem}
	We note however that Lemma \ref{realcknotsr1} is not true in general for \textit{complex} $C(K)$-spaces; indeed, $[0,1]$ is one-dimensional, so by \cite[Proposition~1.7]{rieffel} $C([0,1], \mathbb{C})$ has stable rank one.
\end{Rem}

\begin{Prop}
	Let $K$ be a connected Koszmider space. Then $\mathcal{B}(C(K, \mathbb{R}))$ is Dedekind-finite but it does not have stable rank one.
\end{Prop}

\begin{proof}
	Assume towards a contradiction that $\mathcal{B}(C(K, \mathbb{R}))$ has stable rank one. Then \\
$\mathcal{B}(C(K, \mathbb{R})) / \mathcal{S}(C(K, \mathbb{R}))$ also has stable rank one, which in view of Theorem \ref{koszmiderweakcalkin} is equivalent to $C(K, \mathbb{R})$ having stable rank one. This is impossible by Lemma \ref{realcknotsr1}. \\
	The fact that $\mathcal{B}(C(K, \mathbb{R}))$ is DF follows from Lemma \ref{dfequiv} and Theorem \ref{opalgofckoszmiderdf}.
\end{proof}	

\begin{Rem}
	We remark in passing that both the real and complex examples share the following property: Tarbard's space $X_{\infty}$ and $C(K, \mathbb{R})$ (where $K$ is a connected Koszmider space) are both indecomposable, but not hereditarily indecomposable Banach spaces. Indeed, $X_{\infty}$ is indecomposable by \cite[Proposition~4.1.5]{Tarbard} but not hereditarily indecomposable by \cite[Proposition~4.1.4]{Tarbard}. If $K$ is a connected Koszmider space then by \cite[Theorem~2.5]{koszmider1} it follows that $C(K, \mathbb{R})$ is indecomposable. On the other hand, it is well-known that for any infinite compact Hausdorff space $K$, $C(K)$ cannot be hereditarily indecomposable. This follows from the fact that if $K$ is such then $C(K)$ has a closed subspace (isometrically) isomorphic to $c_0$, see for example, \cite[Lemma~2.5(d)]{rosenthal}.
\end{Rem}		

\begin{Ack}
The author is grateful to his supervisors Dr Yemon Choi and Dr Niels J. Laustsen (Lancaster) for their invaluable advice during the preparation of this note. He is indebted to Dr Tomasz Kania (Prague), Professor Piotr Koszmider (Warsaw) and the anonymous referee for the encouragement and for carefully reading the first version of this note. He also acknowledges the financial support from the Lancaster University Faculty of Science and Technology and acknowledges with thanks the partial funding received from GA\v{C}R project 19-07129Y; RVO 67985840 (Czech Republic). 
\end{Ack}
	
% % % % % % % % % % % % % % % % % % % % % % % % % % % % % % % % % % % % % % %

%\bibliographystyle{plain}
%\bibliography{kitoltendo}

\Addresses

\end{document}